\documentclass[10pt]{amsart}

\usepackage{comment, amsmath, amssymb, amsgen, amsthm, amscd, xspace, color, epsfig, float, epic, multirow, array, lscape}
\usepackage[hypertex]{hyperref}
\usepackage{extarrows}

\makeatletter

\newcommand{\Rmnum}[1]{\expandafter\@slowromancap\romannumeral #1@}

\newcommand{\Ext}{\operatorname{Ext}}

\newcommand{\Hom}{\operatorname{Hom}}

\newcommand{\ch}{\operatorname{ch}}

\newcommand{\rank}{\mathrm{rank}}

\newcommand{\lera}{\leftrightarrow}

\newcommand{\frb}{\mathfrak{b}}

\newcommand{\frg}{\mathfrak{g}}
\newcommand{\frh}{\mathfrak{h}}

\newcommand{\frl}{\mathfrak{l}}

\newcommand{\frp}{\mathfrak{p}}

\newcommand{\fru}{\mathfrak{u}}

\newcommand{\frsl}{\mathfrak{sl}}

\newcommand{\bbC}{\mathbb{C}}

\newcommand{\bbQ}{\mathbb{Q}}

\newcommand{\bbZ}{\mathbb{Z}}

\newcommand{\caO}{\mathcal{O}}

\newtheorem{theorem}[equation]{Theorem}

\newtheorem{prop}[equation]{Proposition}
\newtheorem{lemma}[equation]{Lemma}

\theoremstyle{remark}
\newtheorem{remark}[equation]{Remark}

\theoremstyle{definition}
\newtheorem{definition}[equation]{Definition}

\numberwithin{equation}{section}
\begin{document}

\title[blocks of category $\caO^\frp$]{Blocks of the category $\caO^\frp$ in type $E$}

\author{Jieren Hu}
\address[Hu]{College of Computer Science and Software Engineering, Shenzhen University,
Shenzhen, 518060, Guangdong, P. R. China}
\email{hujieren97@\,qq.com}

\author{Wei Xiao}
\address[Xiao]{College of Mathematics and statistics, Shenzhen Key Laboratory of Advanced Machine Learning and Applications, Shenzhen University,
Shenzhen, 518060, Guangdong, P. R. China}
\email{xiaow@szu.edu.cn}

\author{Ailin Zhang}
\address[Zhang]{College of Mathematics and statistics, Shenzhen Key Laboratory of Advanced Machine Learning and Applications, Shenzhen University,
Shenzhen, 518060, Guangdong, P. R. China}
\email{az304@szu.edu.cn}

\thanks{The second author is supported by the National Science Foundation of China (Grant No. 11701381) and Guangdong Natural Science Foundation (Grant No. 2017A030310138). The third author is supported by the National Science Foundation of China (Grant No. 11504246).}




\bigskip

\begin{abstract}
In this paper, we determine the blocks of $\caO^\frp$ associated with semisimple Lie algebras of type $E$.
\end{abstract}

%

\maketitle

%
%
\section{Introduction}
%
%

Suppose that $\frg$ is a complex semisimple Lie algebra with a Borel subalgebra $\frb$ and a Cartan subalgebra $\frh\subset\frb$. Let $\frp\supset\frb$ be a standard parabolic subalgebra of $\frg$. If two simple modules of $\caO^\frp$ \cite{R, H2} extend nontrivially, they belong to the same block of $\caO^\frp$ \cite{H2}. The equivalence relation generated by this relation partitions the simple modules into blocks. The category $\caO^\frp$ can be decomposed as follows:
\[
\caO^\frp=\bigoplus_{\chi}\caO_{\chi}^\frp,
\]
where $\chi=\chi_\lambda$ for some $\lambda\in\frh^*$ is the infinitesimal character of the full subcategory $\caO_\lambda^\frp:=\caO_{\chi_\lambda}^\frp$. If $\frp=\frb$ or $\lambda$ is regular, then $\caO_\lambda^\frp$ is a block \cite{H2}. Brundan shows that this also holds when the root systems of irreducible components of $\frg$ is of type $A$ \cite{Br}. When $\Phi$ is not of type $A$, there are examples such that $\caO_\lambda^\frp$ is a direct sum of more than one block \cite{ES, BN, P1, X2} (In some paper, the category $\caO_\lambda^\frp$ is called a ``block'' of $\caO^\frp$ and a block is called a ``linkage class'' \cite{P1}). This makes the problem of blocks to be quite subtle.

Theoretically, the extensions between simple modules depend on the leading coefficients (the $\mu$-function) of Kazhdan-Lusztig polynomials \cite{KL}, which can be calculated by a recursion process. Along this line, lower rank cases $F_4$ and $G_2$ was solved using computer programs \cite{P1}. When the rank of $\frg$ is large, the computation of $\mu$-functions turns out to be very difficult \cite{Lu, Xi, LX} since the recursion formulae quickly become unusable. In \cite{X2}, the second authors show that Jantzen coefficients \cite{XZ} also determine the blocks. This Jantzen coefficients (which is relatively easy to calculate) come from the well-known Jantzen filtration \cite{J1, J2} for standard modules. They can be used to obtain the radical filtrations of generalized Verma modules in many cases \cite{HX}. Applying the theory of Jantzen coefficients, we determine the block decomposition of $\caO^\frp$ in this paper for semisimple algebras of type $E$.

The problem of blocks has a deep relation with some other problems about $\caO^\frp$, including simplicity of generalized Verma modules \cite{J2, He, HKZ, BX} as well as homomorphism between generalized Verma modules \cite{Bo, BC, BEJ, BN, L1, L2, M1, M2, M3, X1} and representation types of blocks of $\caO^\frp$ \cite{BN, P2}.

This paper is organized as follows. In section 2, we provide necessary notations and definitions. The relation between Jantzen coefficients and blocks is discovered in section 3. In section 4, we give the algorithm to compute blocks. An example is described in section 5 to illustrate our algorithm. The full data about blocks of type $E$ is presented in section 6.

%
%
\section{Notations and definitions}
%
%

Denote by $\Phi\subset\frh^*$ the root system of $(\frg, \frh)$ with a positive system $\Phi^+$ and a simple system $\Delta\subset\Phi^+$ corresponding to $\frb$.
Let $\Phi_I$ be a subsystem of $\Phi$ associated with $I\subset\Delta$. Let $W$ (resp. $W_I$) be the Weyl group of $\Phi$ (resp. $\Phi_I$). Denote by $l(\cdot)$ the length function on $W$. For $\alpha\in\Phi$ and $\lambda\in\frh^*$, write $s_\alpha\lambda=\lambda-\langle\lambda, \alpha^\vee\rangle\alpha$, where $\langle\cdot, \cdot\rangle$ is the bilinear form on $\frh^*$ induced from the Killing form and $\alpha^\vee:=2\alpha/\langle\alpha, \alpha\rangle$ is the coroot of $\alpha$. The weight $\lambda\in\frh^*$ is called \emph{regular} if $\langle\lambda, \alpha^\vee\rangle\neq0$ for all roots $\alpha\in\Phi$. We say $\lambda$ is \emph{integral} if $\langle\lambda, \alpha^\vee\rangle\in\bbZ$ for all $\alpha\in\Phi$. An integral weight $\lambda\in\frh^*$ is \emph{dominant} if $\langle\lambda, \alpha^\vee\rangle\in\bbZ^{\geq0}$ (resp. $\langle\lambda, \alpha^\vee\rangle\in\bbZ^{\leq0}$) for all $\alpha\in\Delta$. If $\lambda$ is integral, there exists a unique dominant weight $\overline\lambda$ in the orbit $W\lambda$.

Let $\frp$ be the standard parabolic subalgebra corresponding to $I$. Then $\frp=\frl\oplus\fru$, where $\frl$ is a Levi subalgebra and $\fru$ is a nilpotent radical of $\frg$. Set
\[
\Lambda_I^+:=\{\lambda\in\frh^*\ |\ \langle\lambda, \alpha^\vee\rangle\in\bbZ^{>0}\ \mbox{for all}\ \alpha\in I\}.
\]
For $\lambda\in\Lambda_I^+$, let $F(\lambda-\rho)$ be a finite dimensional simple $\frl$-modules of highest weight $\lambda-\rho$, where $\rho:=\frac{1}{2}\sum_{\alpha\in\Phi^+}\alpha$. The generalized Verma module is defined by
\[
M_I(\lambda):=U(\frg)\otimes_{U(\frp)}F(\lambda-\rho).
\]
Here $F(\lambda-\rho)$ is viewed as a $\frp$-module with trivial $\fru$-action. The generalized Verma module $M_I(\lambda)$ and its simple quotients $L(\lambda)$ share the same infinitesimal character $\chi_\lambda$, where $\chi_\lambda$ is an algebra homomorphism from the center $Z(\frg)$ of $U(\frg)$ to $\bbC$ so that $z\cdot v=\chi_\lambda(z)v$ for all $z\in Z(\frg)$ and all $v\in M_I(\lambda)$. If $I=\emptyset$, then $\frp=\frb$ and $\caO^\frp$ collapses to the well known BGG category $\caO$ \cite{BGG}. Moreover, $M(\lambda):=M_\emptyset(\lambda)$ is the Verma module with highest weight $\lambda-\rho$. Denote by $\ch M$ the formal character of module $M\in\caO$. The module $M$ has a composition series with simple quotients isomorphic to some $L(\lambda)$. Denote by $[M : L(\lambda)]$ the multiplicity of $L(\lambda)$.

For $\lambda\in\frh^*$, set
\[
\Phi_\lambda:=\{\beta\in\Phi\mid\langle\lambda,\beta\rangle=0\}.
\]
It is obvious that $\Phi_\lambda$ is a subsystem of $\Phi$. In particular, if $\lambda$ is integral, then $\Phi_\lambda=w\Phi_J$ for some $w\in W$, where $J=\{\alpha\in\Delta\mid\langle\overline\lambda, \alpha\rangle=0\}$. Define
\[
{}^IW^J=\{w\in W\mid \ell(xwy)=\ell(x)+\ell(w)+\ell(y)\ \mbox{for any}\ x\in W_I, y\in W_J\},
\]
Any simple module of $\caO_\lambda^\frp$ can be written as $L(w\overline\lambda)$ for $w\in{}^IW^J$.

%
%
\section{Jantzen coefficients and blocks}
%
%

In this section, we will discuss the relation between Jantzen coefficients and blocks of category $\caO^\frp$. Some results here are already obtained in \cite{X2}, we write down the full proof for self containment.

\subsection{Jantzen filtration and Jantzen coefficients} For $\lambda\in\Lambda_I^+$, we have the following result (see Proposition 9.6 in \cite{H2}):
\begin{equation}\label{jfeq1}
\ch M_I(\lambda)=\sum_{w\in W_I}(-1)^{\ell(w)}\ch M(w\lambda),
\end{equation}
The right side of (\ref{jfeq1}), which we denoted by $\theta(\lambda)$, is a valid character formula for arbitrary $\lambda\in\frh^*$. Denote
\[
\begin{aligned}
\Psi_\lambda^+=&\{\beta\in\Phi^+\backslash\Phi_I\mid \langle\lambda, \beta^\vee\rangle\in\bbZ^{>0}\},\\
\Psi_\lambda^{++}=&\{\beta\in\Psi_\lambda^+\mid \langle s_\beta\lambda, \alpha\rangle\neq0\ \mbox{for all}\ \alpha\in\Phi_I\}.
\end{aligned}
\]

The Jantzen filtration is a useful tool in the study of Verma modules\cite{J1, H2}. With Lemma 3, Satz 2 and the observation in the Bemerkung before Lemma 4 in \cite{J2}, along Jantzen's line for Verma modules (see \cite{J1} or \cite{H2}), one can obtain the following result:

\begin{theorem}[Jantzen filtration and sum formula for generalized Verma modules]\label{jfthm1}
Let $\lambda\in\Lambda_I^+$, then $M_I(\lambda)$ has a filtration by submodules
\[
M_I(\lambda)=M_I(\lambda)^0\supset M_I(\lambda)^1\supset M_I(\lambda)^2\supset\ldots
\]
with $M_I(\lambda)^i=0$ for large $i$, such that
\begin{itemize}
\item [(1)] Every nonzero quotient $M_I(\lambda)^i/M_I(\lambda)^{i+1}$ has a nondegenerate contravariant form.

\item [(2)] $M_I(\lambda)^1$ is the unique maximal submodule of $M_I(\lambda)$.

\item [(3)] There is a formula:
\begin{equation}\label{jft1eq1}
\sum_{i>0}\ch M_I(\lambda)^i=\sum_{\beta\in\Psi_\lambda^+}\theta(s_\beta\lambda).
\end{equation}
\end{itemize}
\end{theorem}

Now recall the Jantzen coefficients defined in \cite{XZ}. The right side of (\ref{jft1eq1}) can be written as
\begin{equation}\label{jfeq2}
\sum_{\beta\in\Psi_\lambda^+}\theta(s_\beta\lambda)=\sum_{\beta\in\Psi_\lambda^{++}}\theta(s_\beta\lambda)=\sum_{\nu\in\Lambda_I^+}c(\lambda, \nu)\ch M_I(\nu),
\end{equation}
where $c(\lambda, \nu)$ are called the \emph{Jantzen coefficients} associated with $(\Phi_I, \Phi)$. Evidently $c(\lambda, \nu)$ are nonzero for only finitely many $\nu\in\Lambda_I^+$.  Write $\mu\leq\lambda$ when $\Hom_\caO(M(\mu), M(\lambda))\neq0$ for $\lambda, \mu\in\frh^*$. If $\lambda, \mu\in\Lambda_I^+$ and $\Ext^1_{\caO^\frp}(L(\mu), L(\lambda))\neq0$, then either $\mu<\lambda$ or $\lambda<\mu$. If $c(\lambda, \mu)=0$, then $\lambda>\mu$.

\begin{definition}\label{jfdef1}
Fix $I\subset\Delta$. Let $\lambda, \mu\in\Lambda_I^+$. We write $\lambda\succeq\mu$ if there exist $\lambda^0, \lambda^1, \ldots, \lambda^{k}\in\Lambda_I^+$ $(k\geq0)$ so that $\lambda^0=\lambda$, $\lambda^{k}=\mu$ and $c(\lambda^{i-1}, \lambda^i)\neq0$ for $1\leq i\leq k$. This induces an ordering on $\Lambda_I^+$. In particular, if $\lambda\succ\mu$ and there exists no $\lambda>\nu>\mu$ with $\lambda\succ\nu\succ\mu$, we say $\lambda$ is \emph{adjacent to} $\mu$.
\end{definition}


\begin{lemma}\label{jflem1}
Let $\lambda, \mu\in\Lambda_I^+$.
\begin{itemize}
\item [(1)] If $[M_I(\lambda), L(\mu)]>0$, then $\lambda\succeq\mu$.
\item [(2)] If $\lambda$ is adjacent to $\mu$, then $[M_I(\lambda), L(\mu)]>0$.
\end{itemize}
\end{lemma}

\begin{proof}
(1) The case $\lambda=\mu$ is trivial. Assume that $[M_I(\lambda), L(\mu)]>0$ and $\mu<\lambda$. We obtain
\begin{equation}\label{jfl1eq1}
\sum_{i>0}[M_I(\lambda)^i : L(\mu)]=\sum_{\nu\in\Lambda_I^+}c(\lambda, \nu)[M_I(\nu) : L(\mu)],
\end{equation}
in view of (\ref{jft1eq1}) and (\ref{jfeq2}). The left side of (\ref{jfl1eq1}) is nonzero. If $c(\lambda, \mu)\neq0$, then $\lambda\succ\mu$. If $c(\lambda, \mu)=0$, there exists $\lambda>\lambda^1>\mu$ such that $c(\lambda, \lambda^1)\neq0$ and $[M_I(\lambda^1) : L(\mu)]\neq0$. Thus $\lambda\succ\lambda^1$. We get (1) when $\lambda^1\succ \mu$, otherwise $c(\lambda^1, \mu)=0$. With $\lambda$ replaced by $\lambda^1$ in the previous argument, there exists $\lambda^1\succ \lambda^2>\mu$ so that $c(\lambda^1, \lambda^2)[M_I(\lambda^2) : L(\mu)]\neq0$. In this spirit, we can eventually get $\lambda\succ\lambda^1\succ\ldots\succ \lambda^{k-1}\succ\mu$ for some $k\geq1$.

(2) Suppose that $\lambda$ is adjacent to $\mu$. Then $c(\lambda, \mu)\neq0$ by definition. If $c(\lambda, \nu)[M_I(\nu) : L(\mu)]\neq0$ for some $\nu\in\Lambda_I^+$, the argument of the first part shows that $\lambda\succ\nu\succeq \mu$. The adjacent condition implies $\nu=\lambda$. So $c(\lambda, \nu)[M_I(\nu) : L(\mu)]=0$ unless $\nu=\mu$. In view of (\ref{jfl1eq1}), one has
\begin{equation}\label{jfl3eq2}
\sum_{i>0}[M_I(\lambda)^i : L(\mu)]=c(\lambda, \mu)\neq0.
\end{equation}
Hence $[M_I(\lambda) : L(\mu)]>0$.
\end{proof}

\subsection{Blocks and lined roots}

\begin{definition}\label{jfdef3}
Fix $I\subset\Delta$. For $\lambda, \mu\in\Lambda_I^+$, write $\lambda\lera\mu$ if $\Ext^1_{\caO^\frp}(L(\mu), L(\lambda))\neq0$. In general, write $\lambda\lera\mu$ if $\lambda=\mu$ or there exist $\lambda^0, \lambda^1, \ldots, \lambda^r\in\Lambda_I^+$ such that
\[
\lambda=\lambda^0\leftrightarrow\lambda^1\leftrightarrow\ldots\leftrightarrow\lambda^r=\mu,
\]
we say $\lambda, \mu\in\Lambda_I^+$ are {\it $\Ext^1$-connected relative to $(\Phi_I, \Phi)$}. For convenience, if $\lambda\lera\mu$ relative to $(\Phi_I, \Phi)$, we write $w_1\lambda\lera w_2\mu$ for any $w_1, w_2\in W_I$.
\end{definition}

The definition gives an equivalence relation on all the $\Phi_I$-regular weights such that each $W_I$ orbit is contained in one equivalence class. The following result is well known.

\begin{lemma}\label{jflem2}
Let $\lambda, \mu\in\Lambda_I^+$ with $\mu<\lambda$. If $\Ext^1_{\caO^\frp}(L(\mu), L(\lambda))\neq0$, then $[M_I(\lambda), L(\mu)]>0$. Conversely, if $[M_I(\lambda), L(\mu)]>0$, then $\lambda\lera\mu$.
\end{lemma}

\begin{lemma}\label{jflem3}
If $c(\lambda, \mu)\neq0$ for $\lambda, \mu\in\lambda_I^+$, there exists $\beta\in\Psi_\lambda^+$ and $w\in W_I$ so that $\mu=ws_\beta\lambda$. Moreover, $\lambda\lera \mu$.
\end{lemma}

\begin{proof}
It suffices to prove the first statement. Note that $c(\lambda, \mu)=\sum_{\beta\in \Psi_{\lambda, \mu}^+}(-1)^{\ell(w_\beta)}$ (see \cite{XZ}, \S4), where
\[
\Psi_{\lambda, \mu}^+:=\{\beta\in\Psi_\lambda^+\mid \mu=w_\beta s_\beta\lambda\ \mbox{for some}\ w_\beta\in W_I\}.
\]
So $\Psi_{\lambda, \mu}^+$ is not empty in view of $c(\lambda, \mu)\neq0$. Choose $\beta\in \Psi_{\lambda, \mu}^+$ and set $w=w_\beta$.
\end{proof}

Lemma \ref{jflem3} and the dual invariance of Jantzen coefficients (see Lemma 4.17 in \cite{XZ}) yield the following result.

\begin{lemma}\label{jflem4}
Let $I, J\subset\Delta$. Suppose that $\lambda$ $($resp. $\mu)$ is a dominant weight with $\Phi_\lambda=\Phi_J$ $($resp. $\Phi_\mu=\Phi_I)$. Choose $x, w\in{}^IW^J$. Then $x\lambda\lera w\lambda$ $($relative to $(\Phi_I, \Phi))$ if and only if $x^{-1}\mu\lera w^{-1}\mu$ $($relative to $(\Phi_J, \Phi))$.
\end{lemma}

The above lemma shows a kind of parabolic-singular duality for category $\caO^\frp$ \cite{So, BGS, B}.
Lemma \ref{jflem3} and the conjugate invariance of Jantzen coefficients (see Lemma 4.18 in \cite{XZ}) implies the following lemma.

\begin{lemma}\label{jflem5}
Let $I, J\subset\Delta$. Suppose that $\Phi_I$ and $\Phi_{J}$ are $W$-conjugate, choose $w\in W$ so that $\Phi_{J}^+=w\Phi_I^+$. Let $\lambda, \mu\in\Lambda_I^+$. Then $\lambda\lera\mu$ $($relative to $(\Phi_{I}, \Phi))$ if and only if $w\lambda\lera w\mu$ $($relative to $(\Phi_{J}, \Phi))$.
\end{lemma}

\begin{remark}\label{bkconrmk1}
In type $E_7$, there are subsets $I, I'\subset\Delta$ such that $\Phi_{I}$ and $\Phi_{I'}$ are of the same type, but they are not $W$-conjugate (see \cite{CM} or \cite{P2}). This happens when $\Phi_I$ is of type $A_5$, $A_3\times A_1$ and $A_1^3$. In each of these cases, there are two conjugate classes, which are denoted by $\Phi'_I$ and $\Phi''_I$.
\end{remark}

\begin{lemma}\label{bkconlem0}
Let $\lambda, \mu\in\Lambda_I^+$. If $\lambda\lera\mu$ relative to $(\Phi_I, \Phi)$, then $\lambda\lera\mu$ relative to $(\Phi_{I'}, \Phi)$ for any $I'\subset I$.
\end{lemma}
\begin{proof}
This follows from the fact that
\[
\Ext_{\caO^{\frp_I}}^1(L(\mu), L(\lambda))\simeq \Ext_{\caO^{\frb}}^1(L(\mu), L(\lambda))\simeq \Ext_{\caO^{\frp_{I'}}}^1(L(\mu), L(\lambda)).
\]
\end{proof}

\begin{definition}\label{bkcondef2}
Suppose $\lambda, \mu\in\Lambda_I^+$. We say $\lambda, \mu$ are {\it linked} $($relative to $(\Phi_I, \Phi))$ when $c(\lambda, \mu)\neq0$. If $\mu=ws_\beta\lambda$ for some $w\in W_I$ and $\beta\in\Psi_\lambda^+$, we say $\beta$ is a {\it linked root} from $\lambda$ to $\mu$ or $\lambda, \mu$ are {\it linked by $\beta$}. In this situation, we will say $w\beta$ is a {\it linked root} from $\mu$ to $\lambda$ (keeping in mind that $\lambda=w^{-1}s_{w\beta}\mu$). In general, if $\beta$ is linked root of $\lambda, \mu\in\Lambda_I^+$, we also say $w_1\beta$ is a linked root from $w_1\lambda$ to $w_2\mu$ (keeping in mind that $w_2\mu=w_2ww_1^{-1}s_{w_1\beta}w_1\lambda$) for $w_1, w_2\in W_I$.
\end{definition}

\begin{lemma}\label{bkconlem5}
For $\lambda, \mu\in\Lambda_I^+$, $\lambda\lera\mu$ if and only if we can find $\lambda^0=\lambda, \lambda^1, \cdots, \lambda^r=\mu\in\Lambda_I^+$ with linked roots $\beta_i$ from $\lambda^{i-1}$ to $\lambda^{i}$ for $1\leq i\leq r$.
\end{lemma}

\begin{proof}
(1) Assume that $\lambda\lera\mu$. By Definition \ref{jfdef3}, it suffices to consider the case $\Ext^1_{\caO^{\frp}}(L(\mu), L(\lambda))\neq0$ with $\mu<\lambda$. In view of Lemma \ref{jflem2}, we have $[M_I(\lambda), L(\mu)]>0$. Lemma \ref{jflem1} yields a chain $\lambda=\lambda^0\succ \lambda^1\succ\cdots\succ \lambda^k=\mu$ such that $\lambda^i\in \Lambda_I^+$ and $c(\lambda, \mu)\neq0$. Then Lemma \ref{jflem3} gives linked roots $\beta_i$ from $\lambda^{i-1}$ to $\lambda^{i}$.

(2) The converse follows from Lemma \ref{jflem3}.
\end{proof}

The Jantzen coefficients have several invariant properties under a reduction process (see \cite{XZ}, \S3). Choose $\beta\in\Phi$. Let $\Phi_{\beta, 0}$ be the irreducible component of $\Phi$ with $\beta\in\Phi_{\beta, 0}$. Set
\[
\begin{aligned}
\Phi_{\beta,1}:=&(\bbQ\Phi_I+\bbQ\beta)\cap\Phi;\\
\Phi_{\beta,2}:=&(\bbQ \Phi_\lambda+\bbQ\beta)\cap\Phi.
\end{aligned}
\]
For $\beta\in\Psi_\lambda^+$, consider the following chain of subsystems
\begin{equation*}
\Phi=\Phi_0(\beta)\supset\Phi_1(\beta)\supset\Phi_2(\beta)\supset\ldots\supset\Phi_m(\beta)\supset\ldots
\end{equation*}
such that for $i\in\bbZ^{>0}$,
\begin{equation*}\label{redseq1}
\Phi_{i+1}(\beta)
=\left\{\begin{aligned}
&(\Phi_{i}(\beta))_{\beta,0}\qquad\qquad\qquad\quad\ \mbox{if}\ i\equiv0\ (\mathrm{mod}4);\\
&(\Phi_{i}(\beta))_{\beta,1}\qquad\qquad\qquad\quad\ \mbox{if}\ i\equiv1\ (\mathrm{mod}4);\\
&(\Phi_{i}(\beta))_{\beta,0}\qquad\qquad\qquad\quad\ \mbox{if}\ i\equiv2\ (\mathrm{mod}4);\\
&(\Phi_{i}(\beta))_{\beta,2}\qquad\qquad\qquad\quad\ \mbox{if}\ i\equiv3\ (\mathrm{mod}4).
\end{aligned}
\right.
\end{equation*}
The chain is stationary, that is, there exists $k\in\bbZ^{>0}$ with $\Phi_k(\beta)=\Phi_{k+1}(\beta)=\ldots$. Denote $\Phi(\beta)=\Phi_k(\beta)$. Recall that for any subsystem $\Phi'$ of $\Phi$, there exists a unique weight $\lambda|_{\Phi'}$ in the subspace $\bbC\Phi'$ so that
\[
\langle\lambda|_{\Phi'}, \alpha^\vee\rangle=\langle\lambda, \alpha^\vee\rangle
\]
for all $\alpha\in\Phi'$.

\begin{lemma}[{\cite[Lemma 3.17]{XZ}}]\label{lrlem1}
Let $\lambda\in\Lambda_I^+$ and $\beta\in\Psi_\lambda^+$. Then $\Phi(\beta)$ is irreducible and $\lambda|_{\Phi(\beta)}$ is an integral weight on $\Phi(\beta)$. Moreover,
\[
\rank(\Phi_I\cap\Phi(\beta))=\rank(\Phi_\lambda\cap\Phi(\beta))=\rank\ \Phi(\beta)-1.
\]
\end{lemma}

In \cite{XZ}, the system $(\Phi, \Phi_I, \Phi_J)$ is called a basic system if $\Phi$ is irreducible and $\rank\Phi_I=\rank\Phi_\lambda=\rank\Phi-1$. The above lemma shows that $(\Phi(\beta), \Phi_I\cap\Phi(\beta), \Phi_\lambda\cap\Phi(\beta))$ is a basic system.

\begin{prop}\label{lrprop2}
Let $\lambda\in\Lambda_I^+$ and $\beta\in\Psi_\lambda^+$. Then $\lambda$ and $s_\beta\lambda$ are linked by $\beta$ $($relative to $(\Phi_I, \Phi))$ if and only if $\lambda|_{\Phi(\beta)}$ and $s_\beta\lambda|_{\Phi(\beta)}$ are linked by $\beta$ $($relative to $(\Phi_I\cap\Phi(\beta), \Phi(\beta)))$.
\end{prop}

\begin{proof}
The invariant properties of Jantzen coefficients (\cite{XZ}, \S4) implies that $s_\beta\lambda$ is $\Phi_I$-regular if and only if $\lambda'=s_\beta(\lambda|_{\Phi(\beta)})=s_\beta\lambda|_{\Phi(\beta)}$ is $\Phi_I\cap\Phi(\beta)$-regular. There exist $\mu\in\Lambda_I^+$ (resp. $\mu'\in\Lambda^+(\lambda|_{\Phi(\beta)}, \Phi_I\cap\Phi(\beta), \Phi(\beta))$) and $w\in W_I$ (resp. $w'\in W(\Phi_I\cap\Phi(\beta))$) with $\mu=ws_\beta\lambda$ (resp. $\mu'=w's_\beta\lambda'$). Moreover, the Jantzen coefficient $c(\lambda, \mu)$ associated with $(\Phi_I, \Phi)$ is nonzero if and only if the Jantzen coefficients $c(\lambda', \mu')$ associated with $(\Phi_I\cap\Phi(\beta), \Phi(\beta))$ is nonzero.
\end{proof}

\begin{lemma}\label{bkconlem6}
Let $\Phi=D_n$ and $\beta\in\Phi^+\backslash\Phi_I$. Suppose that $\{e_{n-1}\pm e_n\}$ is not a subset of $I$. If $\lambda$ and $s_\beta\lambda$ are $\Phi_I$-regular, then $\beta$ is always a linked root from $\lambda$.
\end{lemma}
\begin{proof}
Keeping in mind of Lemma \ref{jflem3}, this lemma is an easy consequence of Theorem 7.27 in \cite{XZ}.
\end{proof}

%
%
\section{The algorithm}
%
%
Although we can run a computer program to determine the blocks of type $E$, it might take too much time when $|I|+|J|$ are too small. For example, if $I=J=\emptyset$, the category $\caO_\lambda^\frp$ has $696729600$ simple modules. This is beyond the processing ability of an ordinary PC. So we need to simplify to calculation. For $\Phi=E_n$ ($n=6, 7, 8$), let $\alpha_1, \cdots, \alpha_n$ be the standard ordering of simple roots (\cite{H1}, \S11.4).

\begin{lemma}\label{slem1}
Let $\Phi=E_n$ for $n=6, 7, 8$. Choose $I\subset\Delta\backslash\{\alpha_1, \alpha_2\}$ or $I\subset\Delta\backslash\{\alpha_1, \alpha_3\}$. Suppose that $\lambda, s_\beta\lambda$ are $\Phi_I$-regular weights with $\beta=e_i\pm e_j$ for $1\leq j<i\leq n-1$. Then $\lambda\lera s_\beta\lambda$.
\end{lemma}
\begin{proof}
This is obvious when $\beta\in\Phi_{I}$. If $\beta\not\in\Phi_I$, set $\Phi':=\Phi_{\Delta'}\simeq D_{n-1}$ and $\lambda'=\lambda|_{\Phi'}$, where $\Delta'=\Delta\backslash\{\alpha_1\}$. Lemma \ref{bkconlem6} implies that $\beta$ is linked root from $\lambda¡¯$. With $\alpha_1\not\in I$, one has $\Phi_{I'}=\Phi_I\cap\Phi'=\Phi_I$. It follows that $\Phi_{\beta, 1}=\Phi'_{\beta, 1}$ and $\Phi(\beta)=\Phi'(\beta)$. In view of Proposition \ref{lrprop2}, $\beta$ is also a linked root from $\lambda$ and thus $\lambda\lera s_\beta\lambda$.
\end{proof}

\begin{lemma}\label{slem2}
Let $\Phi=E_8$. Assume that $|I|+|J|\leq 7$ and $|I|\leq 3$. Suppose that $(\Phi, \Phi_I, \Phi_J)$ contains only one block. If $I'\subset I$, then $(\Phi, \Phi_{I'}, \Phi_{J})$ has only one block.
\end{lemma}
\begin{proof}
Since the case $I=\emptyset$ is trivial, we can assume that $1\leq |I|\leq 3$. This forces $\Phi_I\simeq A_3$, $A_2\times A_1$, $A_1^3$, $A_2$, $A_1^2$ or $A_1$. With Lemma \ref{jflem5}, it suffices to consider the cases $I=\{\alpha_6, \alpha_7, \alpha_8\}$, $\{\alpha_5, \alpha_7, \alpha_8\}$, $\{\alpha_4, \alpha_6, \alpha_8\}$, $\{\alpha_7, \alpha_8\}$, $\{\alpha_6, \alpha_8\}$ or $\{\alpha_8\}$.

For any $\Phi_{I'}$-regular weight with $\Phi_{\overline\lambda}=\Phi_J$, we claim there exists $\hat\lambda\in\Lambda_I^+$ so that $\hat\lambda\lera\lambda$ (relative to $(\Phi_{I'}, \Phi)$). The lemma is an easy consequence of the claim. In fact, given $\lambda, \mu\in{}^{I'}W^J\overline\lambda$, there exist $\Phi_I$-regular weights $\hat\lambda, \hat\mu\in\Lambda_I^+$ so that $\lambda\lera\hat\lambda$ and $\mu\lera\hat\mu$ (relative to $(\Phi_{I'}, \Phi)$). Since $(\Phi, \Phi_I, \Phi_J)$ contains only one block, one has $\hat\lambda\lera\hat\mu$ (relative to $(\Phi_{I}, \Phi)$). Applying Lemma \ref{bkconlem0}, we get $\lambda\lera\hat\lambda\lera\hat\mu\lera\mu$ (relative to $(\Phi_{I'}, \Phi)$), that is, $(\Phi, \Phi_{I'}, \Phi_{J})$ has only one block.

It remains to prove the claim, which can be achieved in a case-by-case fashion. We only discuss the case $I=\{\alpha_6, \alpha_7, \alpha_8\}=\{e_5-e_4, e_6-e_5, e_7-e_6\}$, while the reasoning for the other cases are similar. Set $\Phi':=\Phi_{\Delta'}$ and $\lambda'=\lambda|_{\Phi'}$, where $\Delta'=\Delta\backslash\{\alpha_1\}$. Note that
\[
\rank\Phi'_{\lambda'}=\rank\Phi_\lambda\cap\Phi'\leq \rank \Phi_\lambda=\rank\Phi_J\leq7-|I|=4
\]
in this case. Let $a_1>\cdots>a_m$ be all the different numbers in $\{|\lambda_i|\mid 1\leq i\leq 7\}$. Put $n_i=|\{1\leq i\leq 7\mid |\lambda_i|=a_i\}|$. If $a_m>0$, then $\Phi'_{\lambda'}\simeq A_{n_1-1}\times\cdots\times A_{n_m-1}$, where $A_0=1$. We obtain
\[
\rank\Phi'_{\lambda'}=(n_1-1)+\cdots+ (n_m-1)=7-m\leq 4,
\]
that is, $m\geq3$. If $a_m=0$ and $n_m=1$, similar reasoning shows $m\geq3$. If $a_m=0$ and $n_m\geq2$, then $\Phi'_{\lambda'}\simeq A_{n_1-1}\times\cdots\times D_{n_m}$ and $m\geq4$, where $D_2\simeq A_1\times A_1$ and $D_3\simeq A_3$.

Set $\mu=\lambda$ and $S_\mu=\{|\mu_i|\mid 4\leq i\leq 7\}$. If $a_1\in S_\mu$, there exists $4\leq i_1\leq 7$ so that $|\mu_{i_1}|=a_1$. Otherwise we can find $1\leq i\leq 3$ with $|\mu_i|=a_1$. Thus $s_{e_i-e_7}\mu$ is $\Phi_{I'}$-regular (since $a_1\not\in S_\mu$). Replacing $\mu$ by $s_{e_i-e_7}\mu$, we get $\lambda\lera\mu$ (relative to $(\Phi_{I'}, \Phi)$) and $|\mu_7|=a_1$ ($i_1=7$). In a similar spirit, one can eventually get $\Phi_{I'}$-regular weight $\mu\lera\lambda$ and $|\mu_{i_j}|=a_j$ for $4\leq i_j\leq 7$ and $1\leq j\leq \min\{m, 4\}$. If $\mu_4, \cdots, \mu_7$ are not distinct, one must have $m=3$ and $\mu_{i_4}=\mu_{i_j}=\pm a_j\neq0$ for some $1\leq j\leq 3$, where $i_4\in\{4, 5, 6, 7\}\backslash\{i_1, i_2, i_3\}$. Choose $l\in\{1, 2, 3\}\backslash\{j\}$. Replace $\mu$ by $s_{e_{i_4}+e_{i_l}}\mu$. Then $\mu_4, \cdots, \mu_7$ are distinct and $\lambda\lera\mu$ (relative to $(\Phi_{I'}, \Phi)$). Now we can find $w\in W_I$ so that $w\mu\in\Lambda_I^+$. Since $\mu_4, \cdots, \mu_7$ are distinct, Lemma \ref{slem1} yields $\mu\lera w\mu$ (relative to $(\Phi_{I'}, \Phi)$). Set $\hat\lambda=w\mu$.

\end{proof}

Keeping in mind of Remark \ref{bkconrmk1}, the proof of the following two results are similar and easier.

\begin{lemma}\label{relem2}
Let $\Phi=E_7$. Assume that $|I|+|J|\leq 6$ and $|I|\leq 3$. Suppose that $(\Phi, \Phi_I, \Phi_J)$ contains only one block. If $I'\subset I$, then $(\Phi, \Phi_{I'}, \Phi_{J})$ has only one block.
\end{lemma}

\begin{lemma}\label{relem3}
Let $\Phi=E_6$. Assume that $|I|+|J|\leq 5$ and $|I|\leq 2$. Suppose that $(\Phi, \Phi_I, \Phi_J)$ contains only one block. If $I'\subset I$, then $(\Phi, \Phi_{I'}, \Phi_{J})$ has only one block.
\end{lemma}

\subsection{The algorithm} In this subsection, we will provide an algorithm to find data of the system $(\Phi, \Phi_I, \Phi_J)$.

First consider $\Phi=E_8$. Let $\varpi_J=\sum_{\alpha_j\not\in J}\varpi_j$, where $\varpi_1, \ldots, \varpi_8$ are fundament weights. It is a dominant weight with $\Phi_{\varpi_J}=\Phi_J$. We need to find all the integral weights $\lambda=w\varpi_J$ with $w\in{}^IW^J$. Assume that $\lambda=\sum_{i=1}^8x_i\varpi_k$ for $x_i\in\bbZ$. In particular, $x_i\in \bbZ^{>0}$ when $\alpha_i\in I$ (since $\lambda\in\Lambda_I^+$). Changing basis,
\[
\begin{aligned}
\lambda=&\frac{x_2-x_3}{2}e_1+\frac{x_2+x_3}{2}e_2+(\frac{x_2+x_3}{2}+x_4)e_3+(\frac{x_2+x_3}{2}+x_4+x_5)e_4\\
&+(\frac{x_2+x_3}{2}+x_4+x_5+x_6)e_5+(\frac{x_2+x_3}{2}+x_4+x_5+x_6+x_7)e_6\\
&+(\frac{x_2+x_3}{2}+x_4+x_5+x_6+x_7+x_8)e_7\\
&+(2x_1+\frac{5x_2+7x_3}{2}+5x_4+4x_5+3x_6+2x_7+x_8)e_8.
\end{aligned}
\]
With $\langle\lambda, \lambda\rangle=\langle w\varpi_J, w\varpi_J\rangle=\langle \varpi_J, \varpi_J\rangle$, $x_i$ ($1\leq i\leq 8$) are integral solutions (with $x_i>0$ for $\alpha_i\in I$) of the equation
\[
\begin{aligned}
\langle \varpi_J, \varpi_J\rangle=&(\frac{x_2-x_3}{2})^2+(\frac{x_2+x_3}{2})^2+(\frac{x_2+x_3}{2}+x_4)^2+(\frac{x_2+x_3}{2}+x_4+x_5)^2\\
&+(\frac{x_2+x_3}{2}+x_4+x_5+x_6)^2+(\frac{x_2+x_3}{2}+x_4+x_5+x_6+x_7)^2\\
&+(\frac{x_2+x_3}{2}+x_4+x_5+x_6+x_7+x_8)^2\\
&+(2x_1+\frac{5x_2+7x_3}{2}+5x_4+4x_5+3x_6+2x_7+x_8)^2.
\end{aligned}
\]

The algorithm is summarized as follows.
\begin{itemize}
\item [(1)] Choose $I, J\subset \Delta$ with $|I|+|J|\geq 7$.

\item [(2)] Find all the integral solutions $(x_1, \cdots, x_8)$ of the above equations with $x_i>0$ for $\alpha_i\in I$ using enumeration method.

\item [(3)] For each solution $(x_1, \cdots, x_8)$ obtained in (2), write $\mu=\lambda=\sum_{i=1}^8x_i\varpi_i$. If $\langle\lambda, \alpha\rangle<0$ for some $\alpha\in\Delta$, replace $\mu$ by $s_\alpha\mu$. We eventually arrive at a dominant weight $\mu$. If $\mu=\varpi_J$, keep the solution; otherwise it is discarded.

\item [(4)] Let $\lambda^1, \cdots, \lambda^N$ be all the vector solutions obtained in (3) with $i<j$ whenever $\langle\lambda^i, \varpi_I\rangle>\langle\lambda^j, \varpi_I\rangle$. Set $a(\lambda^i)=i$. First take $\lambda=\lambda^1$. For any $\beta\in\Psi_\lambda^{++}$, we can find $\mu\in\Lambda_I^+$ so that $\mu=ws_\beta\lambda$ for some $w\in W_I$. Then $\mu$ is also a vector solution of (3). If $c(\lambda, \mu)\neq0$ and $a(\lambda)\neq a(\mu)$, set $a(\lambda^i)=\min\{a(\lambda), a(\mu)\}$ for any $1\leq i\leq N$ with $a(\lambda^i)=a(\lambda)$ or $a(\mu)$. Next take $\lambda=\lambda^2, \ldots, \lambda^N$ sequentially and iterate the above procedure.

\item [(5)] When step (4) is finished, let $S$ be the set of all the different $a(\lambda^i)$. Thus the system contains $|S|$ blocks. If $|S|>1$, we output vectors of each block. Otherwise we output the number $N$.
\end{itemize}

It turns out that one always has $|S|=1$ when $|I|+|J|=7$. The following result is an immediate consequence of Lemma \ref{slem2}.

\begin{lemma}\label{aglem1}
Let $\Phi=E_8$. Assume that $|I|+|J|\leq 7$. Then $(\Phi, \Phi_I, \Phi_{J})$ contains only one block.
\end{lemma}

The algorithms for type $E_7, E_6$ are similar, while the equation for type $E_7$ is
\[
\begin{aligned}
\langle \varpi_J, \varpi_J\rangle=&(\frac{x_2-x_3}{2})^2+(\frac{x_2+x_3}{2})^2+(\frac{x_2+x_3}{2}+x_4)^2+(\frac{x_2+x_3}{2}+x_4+x_5)^2\\
&+(\frac{x_2+x_3}{2}+x_4+x_5+x_6)^2+(\frac{x_2+x_3}{2}+x_4+x_5+x_6+x_7)^2\\
&+2(\frac{3x_3+3x_5+x_7}{2}+x_1+x_2+2x_4+x_6)^2
\end{aligned}
\]
and the equation for type $E_6$ is
\[
\begin{aligned}
\langle \varpi_J, \varpi_J\rangle=&(\frac{x_2-x_3}{2})^2+(\frac{x_2+x_3}{2})^2+(\frac{x_2+x_3}{2}+x_4)^2+(\frac{x_2+x_3}{2}+x_4+x_5)^2\\
&+(\frac{x_2+x_3}{2}+x_4+x_5+x_6)^2+3(\frac{x_2}{2}+\frac{2x_1+2x_5+x_6}{3}+\frac{5x_3}{6}+x_4)^2.
\end{aligned}
\]

We also have the following results.

\begin{lemma}\label{aglem2}
Let $\Phi=E_7$. Assume that $|I|+|J|\leq 6$. Then $(\Phi, \Phi_I, \Phi_{J})$ contains only one block.
\end{lemma}

\begin{lemma}\label{aglem3}
Let $\Phi=E_6$. Assume that $|I|+|J|\leq 5$. Then $(\Phi, \Phi_I, \Phi_{J})$ contains only one block.
\end{lemma}

%
%
\section{An example}
%
%
In this section, we give a typical example of system with two blocks. The result is obtained by the algorithm in the previous section. Let $\Phi=E_8$, $I=\{\alpha_2, \alpha_3, \alpha_4, \alpha_5, \alpha_6, \alpha_7\}$ and $J=\{\alpha_3, \alpha_4, \alpha_5, \alpha_7\}$. Thus $\Phi_I\simeq D_6$ and $\Phi_J\simeq A_3\times A_1$. Evidently, $\varpi_J=(\frac{1}{2}, \frac{1}{2}, \frac{1}{2}, \frac{1}{2}, \frac{3}{2}, \frac{3}{2}, \frac{5}{2}, \frac{17}{2})$. All the weights of ${}^IW^J\varpi_J$ are given in Table \ref{extb1}.

\renewcommand\arraystretch{1.3}
\begin{table}[htbp]
\begin{tabular}{|l|c|l|c|}
\hline
$1$ & $(0, 1, 2, 3, 4, 5, 2, 5)$ & $23$ & $(0, 1, 2, 3, 4, 7, 2, -1)$ \\
\hline
$2$ & $(0, 1, 2, 3, 4, 5, -2, 5)$ & $24$ & $(0, 1, 2, 3, 4, 5, 5, -2)$\\
\hline
$3$ & $(0, 1, 2, 3, 4, 6, 3, 3)$ & $25$ & $(-\frac{1}{2}, \frac{3}{2}, \frac{5}{2}, \frac{7}{2}, \frac{9}{2}, \frac{13}{2}, \frac{1}{2}, -\frac{1}{2})$\\
\hline
$4$ & $(\frac{1}{2}, \frac{3}{2}, \frac{5}{2}, \frac{7}{2}, \frac{9}{2}, \frac{11}{2}, \frac{1}{2}, \frac{7}{2})$ & $26$ & $(\frac{1}{2}, \frac{3}{2}, \frac{5}{2}, \frac{7}{2}, \frac{9}{2}, \frac{13}{2}, -\frac{1}{2}, -\frac{1}{2})$\\
\hline
$5$ & $(0, 1, 2, 3, 4, 5, 5, 2)$ & $27$ & $(0, 1, 2, 4, 5, 6, 1, -1)$\\
\hline
$6$ & $(-\frac{1}{2}, \frac{3}{2}, \frac{5}{2}, \frac{7}{2}, \frac{9}{2}, \frac{11}{2}, \frac{5}{2}, \frac{5}{2})$ & $28$ & $(-\frac{1}{2}, \frac{3}{2}, \frac{5}{2}, \frac{7}{2}, \frac{9}{2}, \frac{11}{2}, -\frac{7}{2}, \frac{1}{2})$\\
\hline
$7$ & $(-\frac{1}{2}, \frac{3}{2}, \frac{5}{2}, \frac{7}{2}, \frac{9}{2}, \frac{11}{2}, -\frac{1}{2}, \frac{7}{2})$ & $29$ & $(0, 1, 2, 3, 5, 6, -3, 0)$\\
\hline
$8$ & $(0, 1, 2, 3, 5, 6, 0, 3)$ & $30$ & $(0, 1, 2, 4, 5, 6, -1, -1)$\\
\hline
$9$ & $(0, 1, 2, 3, 4, 7, 1, 2)$ & $31$ & $(\frac{1}{2}, \frac{3}{2}, \frac{5}{2}, \frac{7}{2}, \frac{9}{2}, \frac{11}{2}, \frac{5}{2}, -\frac{5}{2})$\\
\hline
$10$ & $(0, 1, 2, 3, 4, 6, -3, 3)$ & $32$ & $(\frac{1}{2}, \frac{3}{2}, \frac{5}{2}, \frac{7}{2}, \frac{9}{2}, \frac{11}{2}, -\frac{7}{2}, \frac{1}{2})$\\
\hline
$11$ & $(\frac{1}{2}, \frac{3}{2}, \frac{5}{2}, \frac{7}{2}, \frac{9}{2}, \frac{11}{2}, \frac{7}{2}, \frac{1}{2})$ & $33$ & $(0, 1, 2, 3, 4, 7, 1, -2)$\\
\hline
$12$ & $(\frac{1}{2}, \frac{3}{2}, \frac{5}{2}, \frac{7}{2}, \frac{9}{2}, \frac{11}{2}, -\frac{5}{2}, \frac{5}{2})$ & $34$ & $(0, 1, 2, 3, 4, 7, -2, -1)$\\
\hline
$13$ & $(0, 1, 2, 3, 4, 7, 2, 1)$ & $35$ & $(0, 1, 2, 3, 4, 6, 3, -3)$\\
\hline
$14$ & $(0, 1, 2, 3, 4, 7, -1, 2)$ & $36$ & $(0, 1, 2, 3, 4, 7, -1, -2)$\\
\hline
$15$ & $(0, 1, 2, 4, 5, 6, 1, 1)$ & $37$ & $(0, 1, 2, 3, 5, 6, 0, -3)$\\
\hline
$16$ & $(0, 1, 2, 3, 5, 6, 3, 0)$ & $38$ & $(-\frac{1}{2}, \frac{3}{2}, \frac{5}{2}, \frac{7}{2}, \frac{9}{2}, \frac{11}{2}, \frac{1}{2}, -\frac{7}{2})$\\
\hline
$17$ & $(\frac{1}{2}, \frac{3}{2}, \frac{5}{2}, \frac{7}{2}, \frac{9}{2}, \frac{13}{2}, \frac{1}{2}, \frac{1}{2})$ & $39$ & $(-\frac{1}{2}, \frac{3}{2}, \frac{5}{2}, \frac{7}{2}, \frac{9}{2}, \frac{11}{2}, -\frac{5}{2}, -\frac{5}{2})$\\
\hline
$18$ & $(0, 1, 2, 4, 5, 6, -1, 1)$ & $40$ & $(\frac{1}{2}, \frac{3}{2}, \frac{5}{2}, \frac{7}{2}, \frac{9}{2}, \frac{11}{2}, -\frac{1}{2}, -\frac{7}{2})$\\
\hline
$19$ & $(-\frac{1}{2}, \frac{3}{2}, \frac{5}{2}, \frac{7}{2}, \frac{9}{2}, \frac{11}{2}, \frac{7}{2}, -\frac{1}{2})$ & $41$ & $(0, 1, 2, 3, 4, 5, -5, -2)$\\
\hline
$20$ & $(0, 1, 2, 3, 4, 7, -2, 1)$ & $42$ & $(0, 1, 2, 3, 4, 6, -3, -3)$\\
\hline
$21$ & $(0, 1, 2, 3, 4, 5, -5, 2)$ & $43$ & $(0, 1, 2, 3, 4, 5, 2, -5)$\\
\hline
$22$ & $(-\frac{1}{2}, \frac{3}{2}, \frac{5}{2}, \frac{7}{2}, \frac{9}{2}, \frac{13}{2}, -\frac{1}{2}, \frac{1}{2})$ & $44$ & $(0, 1, 2, 3, 4, 5, -2, -5)$\\
\hline
\end{tabular}
\bigskip
\caption{}
\label{extb1}
\end{table}

Table \ref{extb1} shows that the system $(\Phi, \Phi_I, \Phi_J)$ contains $44$ simple highest weight modules. Now we explore the relation between them. Write $c_{i, j}=c(\lambda^i, \lambda^j)$. It was proved in \cite{XZ} that $|c_{i, j}|\leq 1$ in this case. All the nonzero Jantzen coefficients are given in Table \ref{extb2}.

\renewcommand\arraystretch{1.3}
\begin{table}[htbp]
\begin{tabular}{|l|c|c|l|c|c|}
\hline
$i$ & $\{j\mid c_{i, j}=1\}$ & $\{j\mid c_{i, j}=-1\}$ & $i$ & $\{j\mid c_{i, j}=1\}$ & $\{j\mid c_{i, j}=-1\}$\\
\hline
$1$ & $2, 4, 14, 29, 41, 43$ & $8, 20, 32$ & $23$ & $33, 34, 40$ & $37, 44$\\
\hline
$2$ & $7, 9, 16, 24, 44$ & $8, 13, 19$ & $24$ & $41$ & $-$\\
\hline
$3$ & $6, 10, 17, 30, 35, 42$ & $15, 26, 39$ & $25$ & $26, 27, 35$ & $31$\\
\hline
$4$ & $7, 8, 20, 32, 38$ & $14, 29, 41$ & $26$ & $30, 42$ & $39$\\
\hline
$5$ & $11, 21, 23, 24, 37, 44$ & $16, 33, 40$ & $27$ & $30, 31$ & $35$\\
\hline
$6$ & $12, 15, 26, 31, 39$ & $17, 30, 42$ & $28$ & $29, 32, 36, 38$ & $34, 37, 43$\\
\hline
$7$ & $8, 13, 19, 40$ & $9, 16, 24$ & $29$ & $32, 34, 37, 43$ & $36, 38, 41$\\
\hline
$8$ & $9, 14, 16, 24, 29, 37, 41$ & $13, 19, 20, 32$ & $30$ & $39$ & $42$\\
\hline
$9$ & $13, 14, 19, 33$ & $16, 24$ & $31$ & $35, 39$ & $-$\\
\hline
$10$ & $12, 22, 27, 35, 42$ & $18, 25, 31$ & $32$ & $41$ & $-$\\
\hline
$11$ & $16, 19, 28, 33, 40$ & $23, 37, 44$ & $33$ & $36, 37, 44$ & $40$\\
\hline
$12$ & $18, 25, 31, 39$ & $22, 27, 35$ & $34$ & $36, 38$ & $37, 43$\\
\hline
$13$ & $16, 20, 23, 24$ & $19$ & $35$ & $42$ & $-$\\
\hline
$14$ & $20, 32, 36$ & $29, 41$ & $36$ & $37, 43$ & $38$\\
\hline
$15$ & $17, 18, 27, 30, 42$ & $26, 39$ & $37$ & $38, 40$ & $43, 44$\\
\hline
$16$ & $19, 23, 29, 37, 44$ & $24, 33, 40$ & $38$ & $40, 43$ & $-$\\
\hline
$17$ & $22, 25, 26, 39$ & $30, 42$ & $39$ & $42$ & $-$\\
\hline
$18$ & $22, 27, 30, 35$ & $25, 31$ & $40$ & $44$ & $-$\\
\hline
$19$ & $24, 32$ & $-$ & $41$ & $-$ & $-$\\
\hline
$20$ & $29, 34, 41$ & $32$ & $42$ & $-$ & $-$\\
\hline
$21$ & $28, 34, 37, 41, 43$ & $29, 36, 38$ & $43$ & $44$ & $-$\\
\hline
$22$ & $25, 26, 31$ & $27, 35$ & $44$ & $-$ & $-$\\
\hline
\end{tabular}
\bigskip
\caption{}
\label{extb2}
\end{table}

Suppose that $c_{i, j}\neq0$. If there exists no sequence $i=i_0<i_1<\cdots<i_k=j$ such that $c(i_{t-1}, i_{t})\neq0$ for $1\leq t\leq k$ and $k>1$, we say $\lambda^i$ and $\lambda^j$ are \emph{adjacent}. In this case, it was shown in Lemma \ref{jflem1} that $L(\lambda^j)$ is a subquotient of $M_I(\lambda^i)$. We connect $i$ and $j$ when $\lambda^i$ and $\lambda^j$ are adjacent. This gives us the poset in Figure \ref{exfg1}. Two simple modules $L(\lambda^i)$ and $L(\lambda^j)$ belongs to the same block if and only if $i$ and $j$ are connected in the poset. The poset exposes the block decomposition of the system $(\Phi, \Phi_I, \Phi_J)$. The system decomposes into two blocks. One block contains $28$ simple modules and the other contains $16$. It can be easily found in Table \ref{extb2} and Figure \ref{exfg1} that the system has three simple generalized Verma modules $M_I(\lambda^{41})$, $M_I(\lambda^{42})$ and $M_I(\lambda^{44})$. For $i=1, 3, 5$, the simple module $L(\lambda^i)$ is not a composition factor of generalized Verma modules other than $M_I(\lambda^i)$.

\begin{figure}[htbp]\footnotesize
\setlength{\unitlength}{0.9mm}
\begin{center}
\begin{picture}(0,140)

\put(-35,5){\circle{5}} \put(45,5){\circle{5}}
\put(-45,15){\circle{5}} \put(-25,15){\circle{5}} \put(35,15){\circle{5}} \put(55,15){\circle{5}}
\put(-35,25){\circle{5}} \put(35,25){\circle{5}} \put(55,25){\circle{5}}
\put(-35,35){\circle{5}} \put(35,35){\circle{5}} \put(55,35){\circle{5}}
\put(-35,45){\circle{5}} \put(5,45){\circle{5}} \put(45,45){\circle{5}}
\put(-45,55){\circle{5}} \put(-25,55){\circle{5}} \put(-5,55){\circle{5}} \put(15,55){\circle{5}}
\put(45,55){\circle{5}}
\put(-35,65){\circle{5}} \put(-15,65){\circle{5}} \put(5,65){\circle{5}}
\put(35,65){\circle{5}} \put(55,65){\circle{5}}
\put(-45,75){\circle{5}} \put(-25,75){\circle{5}} \put(-5,75){\circle{5}}
\put(35,75){\circle{5}} \put(55,75){\circle{5}}
\put(-55,85){\circle{5}} \put(-35,85){\circle{5}} \put(-15,85){\circle{5}} \put(5,85){\circle{5}}
\put(35,85){\circle{5}} \put(55,85){\circle{5}}
\put(-45,95){\circle{5}} \put(-5,95){\circle{5}} \put(45,95){\circle{5}}
\put(-5,105){\circle{5}}
\put(-5,115){\circle{5}}
\put(-15,125){\circle{5}} \put(5,125){\circle{5}}
\put(-5,135){\circle{5}}

\put(-33.15,6.85){\line(1,1){6.3}} \put(-36.85,6.85){\line(-1,1){6.3}}
\put(-26.85,16.85){\line(-1,1){6.3}} \put(-43.15,16.85){\line(1,1){6.3}}
\put(-35,27.5){\line(0,1){5}}
\put(-35,37.5){\line(0,1){5}}
\put(-33.15,46.85){\line(1,1){6.3}} \put(-36.85,46.85){\line(-1,1){6.3}} \put(6.85,46.85){\line(1,1){6.3}}\put(3.15,46.85){\line(-1,1){6.3}}
\put(-43.15,56.85){\line(1,1){6.3}} \put(-23.15,56.85){\line(1,1){6.3}} \put(-26.85,56.85){\line(-1,1){6.3}} \put(-3.15,56.85){\line(1,1){6.3}}
\put(-6.85,56.85){\line(-1,1){6.3}} \put(13.15,56.85){\line(-1,1){6.3}}
\put(-17.38,65.85){\line(-3,1){25.1}} \put(2.62,65.85){\line(-3,1){25.1}}
\put(-13.15,66.85){\line(1,1){6.3}} \put(-16.85,66.85){\line(-1,1){6.3}}
\put(-33.15,66.85){\line(1,1){6.3}}
\put(-43.15,76.85){\line(1,1){6.3}} \put(-46.85,76.85){\line(-1,1){6.3}}
\put(-23.15,76.85){\line(1,1){6.3}} \put(-26.85,76.85){\line(-1,1){6.3}}
\put(-3.15,76.85){\line(1,1){6.3}} \put(-6.85,76.85){\line(-1,1){6.3}}
\put(-53.15,86.85){\line(1,1){6.3}} \put(-36.85,86.85){\line(-1,1){6.3}} \put(-13.15,86.85){\line(1,1){6.3}} \put(3.15,86.85){\line(-1,1){6.3}}
\put(-5,97.5){\line(0,1){5}}
\put(-5,107.5){\line(0,1){5}}
\put(-3.15,116.85){\line(1,1){6.3}} \put(-6.85,116.85){\line(-1,1){6.3}}
\put(-13.15,126.85){\line(1,1){6.3}} \put(3.15,126.85){\line(-1,1){6.3}}

\put(46.85,6.85){\line(1,1){6.3}} \put(43.15,6.85){\line(-1,1){6.3}}
\put(35,17.5){\line(0,1){5}} \put(55,17.5){\line(0,1){5}}
\put(35,27.5){\line(0,1){5}} \put(55,27.5){\line(0,1){5}}
\put(36.85,36.85){\line(1,1){6.3}} \put(53.15,36.85){\line(-1,1){6.3}}
\put(45,47.5){\line(0,1){5}}
\put(46.85,56.85){\line(1,1){6.3}} \put(43.15,56.85){\line(-1,1){6.3}}
\put(35,67.5){\line(0,1){5}} \put(55,67.5){\line(0,1){5}}
\put(35,77.5){\line(0,1){5}} \put(55,77.5){\line(0,1){5}}
\put(36.85,86.85){\line(1,1){6.3}} \put(53.15,86.85){\line(-1,1){6.3}}

\put(52.7,66.15){\line(-2,1){15.4}}
\put(52.7,76.15){\line(-2,1){15.4}}
\put(37.3,26.15){\line(2,1){15.4}}
\put(37.3,16.15){\line(2,1){15.4}}

\put(-36.7,4){$44$} \put(43.3,4){$42$}
\put(-46.7,14){$40$} \put(-26.7,14){$43$} \put(33.3,14){$39$} \put(53.3,14){$35$}
\put(-36.7,24){$38$} \put(33.3,24){$30$} \put(53.3,24){$31$}
\put(-36.7,34){$37$} \put(33.3,34){$26$} \put(53.3,34){$27$}
\put(-36.7,44){$36$} \put(3.3,44){$41$} \put(43.3,44){$25$}
\put(-46.7,54){$33$} \put(-26.7,54){$34$} \put(-6.7,54){$32$}  \put(13.3,54){$24$}
\put(43.3,54){$22$}
\put(-36.7,64){$23$} \put(-16.7,64){$29$} \put(3.3,64){$19$}
\put(33.3,64){$17$} \put(53.3,64){$18$}
\put(-46.7,74){$28$} \put(-26.7,74){$16$} \put(-6.7,74){$20$}
\put(33.3,74){$15$} \put(53.3,74){$12$}
\put(-56.7,84){$21$} \put(-36.7,84){$11$} \put(-16.7,84){$13$} \put(3.3,84){$14$}
\put(34.2,84){$6$} \put(53.3,84){$10$}
\put(-45.8,94){$5$} \put(-5.8,94){$9$} \put(44.2,94){$3$}
\put(-5.8,104){$8$}
\put(-5.8,114){$7$}
\put(-15.6,124.0){$2$} \put(4.2,124){$4$}
\put(-5.8,134){$1$}

\end{picture}
\end{center}
\caption{}
\label{exfg1}
\end{figure}

%
%
\section{Disconnected systems}
%
%
In this section, we will present the data of disconnected systems, that is, those systems which contains more than one block.

There are $418$ different pairs of $(I, J)$ such that the corresponding system is disconnected for type $E_8$, while there are $294$ pairs for type $E_7$ and $110$ pairs for type $E_6$. Of course we can not describe the full information associated with each pair as the example in the previous section. In view of Lemma \ref{jflem4} and Lemma \ref{jflem5}, the block decomposition of the system $(\Phi, \Phi_I, \Phi_J)$ has some invariant properties up to conjugate classes of $\Phi_I$ and $\Phi_J$. The disconnected systems are listed in Table \ref{dctb1}, Table \ref{dctb3} and Table \ref{dctb5} according to these conjugate classes. By duality, we only illustrated data for those pairs with $|I|\geq |J|$. If $|I|=|J|$ and $I\neq J$, we pick either one of the two pairs. The conjugate classes of $\Phi_I$ are listed in the first and the forth columns in each table, while the conjugate classes of $\Phi_J$ are listed in the second and the fifth columns. If a conjugate class $\Phi_I$ contains $k$ subsets $I$ with $k>1$, we write $\Phi_I(k)$ in the table. For example $D_4\times A_1(2)$ in Table \ref{dctb1} means that there are two subsets $I$ (which are $\{\alpha_2, \alpha_3, \alpha_4, \alpha_5, \alpha_7\}$ and $\{\alpha_2, \alpha_3, \alpha_4, \alpha_5, \alpha_8\}$) such that $\Phi_I$ are in the conjugate class represented by $D_4\times A_1$. The number of blocks and simple modules are given in the third and sixth columns. For example, if $\Phi_I$ and $\Phi_J$ are both conjugate to $D_4\times A_2$, the formula $2\times 12+20$ means that the corresponding systems have three blocks, two of them contain $12$ simple modules, while the other one has $20$ simple modules.

\renewcommand\arraystretch{1.3}
\begin{table}[H]\footnotesize
\begin{tabular}{|c|c|c|c|c|c|}
\hline
$\Phi_I$ & $\Phi_J$ & blocks & $\Phi_I$ & $\Phi_J$ & blocks \\
\hline
$E_7$ & $A_1^3(21)$ & $2\times1$ & $A_5(4)$ & $D_4\times A_1(2)$ & $72+201$\\
\hline
$D_7$ & $A_2\times A_1^3(8)$ & $2\times1$ & $A_5\times A_1(3)$ & $D_4$ & $96+144$\\
\hline
$A_7$ & $A_3\times A_2\times A_1(4)$ & $2\times1$ & $D_4\times A_2$ & $D_4$ & $2\times 192$\\
\hline
$E_6\times A_1$ & $A_3\times A_1(20)$ & $2\times1$ & $A_5(4)$ & $D_4$ & $192+504$\\
\hline
$D_5\times A_1(3)$ & $A_5(4)$ & $2\times1$ & $D_4\times A_1(2)$ & $D_4\times A_1(2)$ & $288+450$\\
\hline
$D_5\times A_2$ & $A_3^2(2)$ & $2\times 2$ & $D_4\times A_1(2)$ & $D_4$ & $576+1224$\\
\hline
$D_6$ & $A_3\times A_2(10)$ & $2\times 2$ & $D_4$ & $D_4$ & $1152+3366$\\
\hline
$D_6$ & $A_2^2\times A_1^2(2)$ & $2\times 4$ & $E_6\times A_1$ & $A_2^2\times A_1(8)$ & $3\times 1$\\
\hline
$E_6\times A_1$ & $A_2^2(8)$ & $2\times 6$ & $D_6$ & $A_3\times A^1_2(10)$ & $3\times 4$\\
\hline
$A_4\times A_3$ & $D_4\times A_2$ & $2\times 6$ & $D_4\times A_2$ & $A_5\times A_1(3)$ & $3\times 6$\\
\hline
$A_4\times A_3$ & $A_5(4)$ & $2\times 6$ & $D_4\times A_2$ & $D_4\times A_2$ & $2\times 12+20$\\
\hline
$E_6$ & $A_2^2(8)$ & $2\times 12$ & $A_5(4)$ & $A_5(4)$ & $2\times 12+20+36$\\
\hline
$D_6$ & $A_3\times A_1(20)$ & $16+28$ & $A_5\times A_1(3)$ & $A_5(4)$ & $5\times 6$\\
\hline
$D_4\times A_2$ & $A_5(4)$ & $12+37$ & $A_4\times A_3$ & $A_5\times A_1(3)$ & $6\times 1$\\
\hline
$A_5\times A_1(3)$ & $D_4\times A_1(2)$ & $36+60$ & $A_4\times A_3$ & $A_4\times A_3$ & $7\times 1$\\
\hline
$E_6$ & $A_3(7)$ & $65+80$ & $A_5\times A_1(3)$ & $A_5\times A_1(3)$ & $12\times 1$\\
\hline
$D_4\times A_2$ & $D_4\times A_1(2)$ & $72+96$ & & &\\
\hline
\end{tabular}
\bigskip
\caption{Disconnected systems of $E_8$}
\label{dctb1}
\end{table}

When $\Phi=E_8$, Table \ref{dctb1} shows that a disconnected system could contain $2$, $3$, $4$, $5$, $6$, $7$ or $12$ blocks. It has at most $4518$ simple modules (see $(E_8, D_4, D_4)$). We say a system is semisimple if each of its block has exactly one simple module.
$9$ of these conjugate pairs corresponding to disconnected semisimple systems, while $9$ conjugate pairs corresponding to connected semisimple systems. All of them are listed in Table \ref{dctb2}.

\renewcommand\arraystretch{1.3}
\begin{table}[H]\footnotesize
\begin{tabular}{|c|c|c|c|c|c|}
\hline
$\Phi_I$ & $\Phi_J$ & blocks & $\Phi_I$ & $\Phi_J$ & blocks \\
\hline
$E_8$ & $\emptyset$ & $1$ & $E_7$ & $A_1^3(21)$ & $2$\\
\hline
$E_7$ & $A_2(7)$ & $1$ & $D_7$ & $A_2\times A_1^3(8)$ & $2$\\
\hline
$D_7$ & $A_2^2(8)$ & $1$ & $A_7$ & $A_3\times A_2\times A_1(4)$ & $2$\\
\hline
$A_6\times A_1$ & $A_4\times A_2\times A_1$ & $1$ & $E_6\times A_1$ & $A_3\times A_1(20)$ & $2$\\
\hline
$D_5\times A_2$ & $A_4\times A_2(4)$ & $1$ & $D_5\times A_1(3)$ & $A_5(4)$ & $2$\\
\hline
$E_6$ & $D_4$ & $1$ & $E_6\times A_1$ & $A_2^2\times A_1(8)$ & $3$\\
\hline
$D_6$ & $A_4(6)$ & $1$ & $A_4\times A_3$ & $A_5\times A_1(3)$ & $6$\\
\hline
$A_6(3)$ & $D_4\times A_2$ & $1$ & $A_4\times A_3$ & $A_4\times A_3$ & $7$\\
\hline
$D_5(2)$ & $D_5(2)$ & $1$ & $A_5\times A_1(3)$ & $A_5\times A_1(3)$ & $12$\\
\hline
\end{tabular}
\bigskip
\caption{Semisimple systems of $E_8$}
\label{dctb2}
\end{table}

When $\Phi=E_7$, Table \ref{dctb3} shows that a disconnected system could contain $2$ or $3$ blocks. It has at most $150$ simple modules (see $(E_7, D_4, A_3)$).
$9$ of these conjugate pairs correspond to disconnected semisimple systems, while $5$ conjugate pairs correspond to connected semisimple systems. All of them are listed in Table \ref{dctb4}.

\renewcommand\arraystretch{1.3}
\begin{table}[H]\footnotesize
\begin{tabular}{|c|c|c|c|c|c|}
\hline
$\Phi_I$ & $\Phi_J$ & blocks &  $\Phi_I$ & $\Phi_J$ & blocks \\
\hline
$D_6$ & $(A_1^3)'(10)$ & $2\times1$ & $D_4\times A_1$ & $(A_3\times A_1)'(9)$ & $8+12$\\
\hline
$A_5\times A_1$ & $(A_3\times A_1)'(9)$ & $2\times1$ & $A_4(5)$ & $A_4(5)$ & $2\times 10$\\
\hline
$(A_5)'(2)$ & $A_3\times A_1^2(3)$ & $2\times1$ & $D_4$ & $(A_3\times A_1)'(9)$ & $2\times 24$\\
\hline
$A_4\times A_1(5)$ & $A_4\times A_1(5)$ & $2\times1$ & $D_4\times A_1$ & $A_3(6)$ & $24+33$\\
\hline
$(A_5)'(2)$ & $A_2^2\times A_1(3)$ & $2\times2$ & $D_4$ & $A_3(6)$ & $48+102$\\
\hline
$A_3\times A_2(3)$ & $D_4\times A_1$ & $2\times 2$ & $A_5\times A_1$ & $A_2^2\times A_1(3)$ & $3\times 1$\\
\hline
$A_4\times A_1(5)$ & $A_4(5)$ & $2\times 3$ & $(A_5)'(2)$ & $(A_3\times A_1)'(9)$ & $3\times 2$\\
\hline
$(A_5)'(2)$ & $A_3(6)$ & $8+12$ & & &\\
\hline
\end{tabular}
\bigskip
\caption{Disconnected systems of $E_7$}
\label{dctb3}
\end{table}

\renewcommand\arraystretch{1.3}
\begin{table}[H]\footnotesize
\begin{tabular}{|c|c|c|c|c|c|}
\hline
$\Phi_I$ & $\Phi_J$ & blocks &  $\Phi_I$ & $\Phi_J$ & blocks \\
\hline
$E_7$ & $\emptyset$ & $1$ & $D_4\times A_1$ & $A_4(5)$ & $1$\\
\hline
$E_6$ & $(A_1^3)''$ & $1$ & $A_4\times A_2$ & $A_3\times A_2\times A_1$ & $1$\\
\hline
$D_6$ & $A_2(6)$ & $1$ & $D_6$ & $(A_1^3)'(10)$ & $2$\\
\hline
$A_6$ & $A_2\times A_1^3$ & $1$ & $A_5\times A_1$ & $(A_3\times A_1)'(9)$ & $2$\\
\hline
$D_5\times A_1$ & $A_2^2(4)$ & $1$ & $(A_5)'(2)$ & $A_3\times A_1^2(3)$ & $2$\\
\hline
$D_5(2)$ & $(A_3\times A_1)''(2)$ & $1$ & $A_4\times A_1(5)$ & $A_4\times A_1(5)$ & $2$\\
\hline
$(A_5)''$ & $D_4$ & $1$ & $A_5\times A_1$ & $A_2^2\times A_1(3)$ & $3$\\
\hline
\end{tabular}
\bigskip
\caption{Semisimple systems of $E_7$}
\label{dctb4}
\end{table}

When $\Phi=E_6$, Table \ref{dctb3} shows that a disconnected system could contain $2$ or $3$ blocks. It has at most $25$ simple modules (see $(E_6, A_3, A_3)$).
$6$ of these conjugate pairs correspond to disconnected semisimple systems, while $4$ conjugate pairs correspond to connected sesimple systems. All of them are listed in Table \ref{dctb4}.

\renewcommand\arraystretch{1.3}
\begin{table}[H]\footnotesize
\begin{tabular}{|c|c|c|c|c|c|}
\hline
$\Phi_I$ & $\Phi_J$ & blocks &  $\Phi_I$ & $\Phi_J$ & blocks \\
\hline
$A_5$ & $A_1^3(5)$ & $2\times1$ & $A_3(5)$ & $A_3(5)$ & $8+17$\\
\hline
$A_2^2\times A_1$ & $A_3\times A_1(4)$ & $2\times1$ & $A_2^2\times A_1$ & $A_2^2\times A_1$ & $3\times 1$\\
\hline
$A_3\times A_1(4)$ & $A_3(5)$ & $2\times4$ & $A_3\times A_1(4)$ & $A_3\times A_1(4)$ & $3\times 1$\\
\hline
$D_4$ & $A_2(5)$ & $2\times6$ & & &\\
\hline
\end{tabular}
\bigskip
\caption{Disconnected systems of $E_6$}
\label{dctb5}
\end{table}

\renewcommand\arraystretch{1.3}
\begin{table}[H]\footnotesize
\begin{tabular}{|c|c|c|c|c|c|}
\hline
$\Phi_I$ & $\Phi_J$ & blocks &  $\Phi_I$ & $\Phi_J$ & blocks \\
\hline
$E_6$ & $\emptyset$ & $1$ & $A_4(4)$ & $A_3(5)$ & $1$\\
\hline
$D_5(2)$ & $A_1^2(10)$ & $1$ & $A_5$ & $A_1^3(5)$ & $2$\\
\hline
$A_5$ & $A_2(5)$ & $1$ & $A_2^2\times A_1$ & $A_3\times A_1(4)$ & $2$\\
\hline
$A_4\times A_1(2)$ & $A_2\times A_1^2(5)$ & $1$ & $A_2^2\times A_1$ & $A_2^2\times A_1$ & $3$\\
\hline
$D_4$ & $A_2^2$ & $1$ & $A_3\times A_1(4)$ & $A_3\times A_1(4)$ & $3$\\
\hline
\end{tabular}
\bigskip
\caption{Semisimple systems of $E_6$}
\label{dctb6}
\end{table}

\end{document}